\documentclass[11pt]{article}
\usepackage{latexsym,amsmath,color,amsthm,amssymb,epsfig,graphicx,mathrsfs,comment}
\usepackage{enumitem,graphicx}
\usepackage[linktocpage=true]{hyperref}
\usepackage{setspace}
\usepackage[toc,page]{appendix}
\usepackage[left=1in,top=1in,right=1in,bottom=1in]{geometry}
\usepackage{bm}
\usepackage[compress]{cite}
\makeatletter
\def\thm@space@setup{\thm@preskip=12pt \thm@postskip=8pt}

\newcommand{\N}{\mathbb{N}}

\newcommand{\A}{\mathcal{A}}
\newcommand{\B}{\mathcal{B}}

\newcommand{\F}{\mathcal{F}}
\newcommand{\HM}{\mathcal{H}}
\newcommand{\h}{\mathcal{H}}

\newtheorem{thm}{Theorem}

\newtheorem{prop}{Proposition}
\newtheorem{lemma}{Lemma}

\newtheorem{conj}{Conjecture}

\theoremstyle{definition}

\theoremstyle{remark}

\begin{document}
	
\title{Non-trivial $d$-wise Intersecting families }
\author{
	Jason O'Neill \\
	Department of Mathematics, \\
	University of California, San Diego \\
	\texttt{jmoneill@ucsd.edu}
	\and
	Jacques Verstra\"{e}te
	\footnote{Research supported by NSF award DMS-1800332} \\
	Department of Mathematics, \\
	University of California, San Diego \\
	\texttt{jverstra@math.ucsd.edu}
}

\maketitle

\begin{abstract}
For an integer $d \geq 2$, a family $\F$ of sets is \textit{$d$-wise intersecting} if for any distinct sets $A_1,A_2,\dots,A_d \in \F$, $A_1 \cap A_2 \cap \dots \cap A_d \neq \emptyset$, and \textit{non-trivial} if $\bigcap \F = \emptyset$. Hilton and Milner conjectured that for $k \geq d \geq 2$ and large enough $n$, the extremal non-trivial $d$-wise intersecting family of $k$-element subsets of $[n]$ is one of the following two families:
\begin{eqnarray*}
 \HM(k,d) &=&  \{A \in \textstyle{{[n] \choose k}} : [d-1] \subset A, A \cap [d,k+1] \neq \emptyset\} \cup \{[k+1] \setminus \{i \}  : i \in [d - 1]\} \\
 \A(k,d) &=& \{ A \in \textstyle{{[n] \choose k}} :  |A \cap [d+1]| \geq d   \}.
  \end{eqnarray*}
 The celebrated Hilton-Milner Theorem states that $\HM(k,2)$ is the unique extremal non-trivial intersecting family for $k>3$. 
 We prove the conjecture and prove a stability theorem, stating that any large enough non-trivial $d$-wise intersecting family of $k$-element subsets of $[n]$ 
 is a subfamily of $\A(k,d)$ or $\HM(k,d)$. 
 \end{abstract}	

\section{Introduction}

The celebrated Erd\H{o}s-Ko-Rado Theorem~\cite{EKR} states that for $n \geq 2k$, the maximum size of an intersecting family of $k$-element subsets of $[n] := \{1,2,\dots,n\}$ is 
${n - 1 \choose k - 1}$. Furthermore, equality holds for $n > 2k$ if and only if there is a point in the intersection of all sets in the family.
Here and in what follows we write ${[n] \choose k}$ for the family of $k$-element subsets of $[n]$, and $[a,b] = \{a,a+1,\dots,b\}$ for integers $a < b$. 
 In their paper, Erd{\H o}s, Ko and Rado asked for the maximum 
size of an intersecting family $\F$ of $k$-element subsets of $[n]$ such that $\bigcap \F = \emptyset$. This question was answered by Hilton and Milner \cite{HM}. 

\begin{thm}[Hilton-Milner]\label{thm:HM}
Let $n > 2k$ and $k \geq 3$. If $\F \subset \binom{[n]}{k}$ is a non-trivial intersecting family, then $|\F| \leq \binom{n-1}{k-1}-\binom{n-k-1}{k-1}+1$.
\end{thm}

This may be viewed as a stability version of the Erd\H{o}s-Ko-Rado Theorem, in the sense that an intersecting family of size larger than the bound in Theorem \ref{thm:HM} 
is necessarily a subfamily of the extremal intersecting family -- i.e. there a point in the intersection of all sets in the family.  There have been many recent directions~\cite{BD,HY,K,FHHZ,FS,GS} in the classical Hilton-Milner theory generalizing Theorem \ref{thm:HM}. Hilton and Milner~\cite{HM} considered an extension of Theorem \ref{thm:HM} to $d$-wise intersecting families: a family $\F$ of sets is \textit{$d$-wise intersecting} if any set of $d$
distinct sets in $\F$ have non-empty intersection, with the case $d = 2$ corresponding to intersecting families. Hilton and Milner~\cite{HM} conjectured that if $\F$ is a $d$-wise intersecting family of $k$-sets in $[n]$,  then for large enough $n$, one of the following two families is extremal:
\begin{eqnarray*}
 \A(k,d) &=& \{ A \in \textstyle{{[n] \choose k}} :  |A \cap [d+1]| \geq d   \} \\
 \HM(k,d) &=& \{A \in \textstyle{{[n] \choose k}} : [d-1] \subset A, A \cap [d,k+1] \neq \emptyset\} \cup \{[k+1] \setminus \{i \}  : i \in [d - 1]\}. 
 \end{eqnarray*}
 We suppress notation to indicate that $\A(k,d)$ and $\HM(k,d)$ depend on $n$. 
It is straightforward to check for large enough $n$ relative to $k$ and $d$ that $|\A(k,d)| \geq |\HM(k,d)|$ if and only if $2d \geq k + 1$.
In fact, the sizes of these families are given by
\begin{eqnarray*}
|\A(k,d)| &=&  (d+1) \binom{n-d-1}{k-d} + \binom{n-d-1}{k-d-1}  \\
|\HM(k,d)|  &=&  \binom{n-d+1}{k-d+1} - \binom{n-k-1}{k-d+1} + d - 1.
\end{eqnarray*} 

In this paper, we prove the conjecture of Hilton and Milner, including a stability result.  To state our theorem, we need the following additional non-trivial $d$-wise intersecting family:
\[  \B(k,d) = \{B \in \textstyle{{[n] \choose k}}: |B \cap [d-1]|=d-2, [d,k] \subset B\} \bigcup \; \{ B \in \textstyle{{[n] \choose k}} : [d-1] \subset B,  B \cap [d,k] \neq \emptyset\}. \]
The role of this family is in the stability for non-trivial $d$-wise intersecting families of $k$-element sets when $2d < k$, in which case $|\A(k,d)| \leq |\B(k,d)| \leq |\HM(k,d)|$ with equalities if and only if $k = 3$ and $d = 2$ and are all isomorphic when $d=k$. Our main theorem is as follows:

\begin{thm}\label{thm:mainthm1}
Let $k,d$ be integers with $2 \leq d < k$. Then there exists $n_0(k,d)$ such that for $n \geq n_0(k,d)$, if $\F$ is a non-trivial, $d$-wise intersecting family
of $k$-element subsets of $[n]$, then
\[ |\F| \leq \max\{ |\HM(k,d)|, |\A(k,d)| \}.\]
Furthermore, if $2d \geq k$ and $|\F| > \min\{|\HM(k,d)|,|\A(k,d)|\}$, then $\F \subseteq \HM(k,d)$ or $\F \subseteq \A(k,d)$. If $2d < k$ and $|\F| > |\B(k,d)|$, 
then $\F \subseteq \HM(k,d)$.
\end{thm}

We use the Delta system method to prove Theorem \ref{thm:mainthm1}, which gives $n_0(k,d) = d + e(k^2 2^k)^{2^k}(k-d)$. This is very unlikely to be best possible, 
and we conjecture that the following holds:

\begin{conj}\label{conj:mainconj}
For $k > d \geq 2$ and $n \geq kd/(d - 1)$, the unique extremal non-trivial $d$-wise intersecting families of $k$-element subsets of $[n]$ are 
$\HM(k,d)$ and $\A(k,d)$.
\end{conj}

The structure of this paper is as follows. In Section \ref{sec:basicprop}, we establish basic properties and structural results. In Section \ref{sec:deltasystemmethod}, we describe the Delta system method and apply it to non-trivial $d$-wise intersecting families. We then prove Theorem \ref{thm:mainthm1} in Section \ref{sec:proof}. We will use calligraphic font to denote set families, capital letters to denote sets and lower case letters to denote elements.

\section{Preliminaries}\label{sec:basicprop}

In this section, we will prove some basic facts and structural results pertaining to non-trivial $d$-wise intersecting families. 

\subsection{Basic Properties}

We will first show, as was initially done by Hilton and Milner in \cite{HM}, that there cannot be a $k$-uniform non-trivial $d$-wise intersecting family for $d>k$. 

\begin{lemma}\label{lemma:largeintersections}\textup{\cite{HM}}
Let $d>k$, then there does not exist a $d$-wise intersecting non-trivial $\F \subset \binom{[n]}{k}$.
\end{lemma}

\begin{proof}
Fix $A \in \F$, then for each $a \in A$, there exists $X_a \in \F$ so that $a \notin X_a$ by the definition of non-trivial. Then $A \cap \bigcap_{a \in A} X_a = \emptyset$ which is a contradiction.
\end{proof}

A similar argument as in Lemma \ref{lemma:largeintersections} also gives an upper-bound on the $m$-wise intersections from a non-trivial $d$-wise intersecting family. 

\begin{lemma}\label{lemma:mwiseintersections} 
Given a non-trivial $d$-wise intersecting family $\F \subset \binom{[n]}{k}$ and $A_1, \ldots , A_m \in \F$,$$ |\bigcap\limits_{i=1}^m A_i| \geq d-(m-1).$$
\end{lemma}

\begin{proof}
Suppose not, then we may find a set in $\F$ for each element in the above intersection which does not contain that element and violate the $d$-wise intersecting property of $\F$. 
\end{proof}

In the case where $d=k$, we first note that $K_{k+1}^{(k)}$ is a non-trivial $k$-wise intersecting family and prove this is the only such example.

\begin{prop}\label{thm:dequalsk}\textup{\cite{HM}}
If $n \geq k+1$ and $\F \subset \binom{[n]}{k}$ is a $k$-wise intersecting non-trivial family, then $\F \cong K_{k+1}^{(k)} \cong \A(k,k)$.
\end{prop}

\begin{proof}
Let $A \in \F$, then without loss of generality we may assume that $A = [k]$. Observe that there exists $A_1 \in \F$ so that $1 \notin A_1$ and by Lemma \ref{lemma:mwiseintersections}, $|A \cap A_1|=k-1$. Without loss of generality, let $A_1 \setminus A = \{k+1\}$. Then for each $i \in [2,k]$, there exists $A_i \in \F$ so that $i \notin A_i$. Next, Lemma \ref{lemma:mwiseintersections} yields that $|A \cap A_i|=k-1$ and $|A_1 \cap A_i| = k-1$ and as a result $A_i \cap [k+2,n] = \emptyset$. Putting these all together we get that $K_{k+1}^{(k)} \subseteq \F$ and noting that $K_{k+1}^{(k)}$ is saturated yields the desired result.
\end{proof}

Proposition \ref{thm:dequalsk} verifies Conjecture \ref{conj:mainconj} in the case where $d=k$.

\subsection{Structure of non-trivial $d$-wise Intersecting Families}\label{sec:structuredeltarich}

Following the notation from Mubayi and the second author \cite{MV}, a \textit{Delta system} is a hypergraph $\Delta$ such that for all distinct $e, f \in \Delta$, $ e \cap f = \cap_{g \in \Delta} g$. We let $\Delta_{k,s}$ be a $k$-uniform Delta system with $s$ edges and define $\text{core}(\Delta):= \cap_{g \in \Delta} g$. Let $\F \subset \binom{[n]}{k}$ and $X \subset [n]$, then define the \textit{core degree} of $X$ in $\F$ to be $$ d_{\F}^\star(X) := \max \{ s: \exists \; \Delta_{k,s} \text{ so that } \text{core}(\Delta_{k,s}) = X   \}. $$

\medskip
In this section, we will examine the collection of $d$-sets with large core degree with respect to a non-trivial $d$-wise intersecting family. We will show that this collection of $d$-sets is necessarily isomorphic to a subfamily of one of the corresponding collections of $d$-sets in the extremal examples $\HM(k,d)$ and $\A(k,d)$. Moreover, given enough $d$-sets with large core degree, we show that $|\F|$ is less than or equal to $ \max\{ |\A(k,d)|, |\HM(k,d)| \}$.

\medskip 
By Lemma \ref{lemma:mwiseintersections}, $|A \cap B| \geq d-1$ for all $A,B \in \F$ in a non-trivial $d$-wise intersecting family $\F$, and hence $d_{\F}^\star(X) \leq 1$ whenever $|X|<d-1$. We will now show that $(d-1)$-sets cannot have large core degree in non-trivial $d$-wise intersecting families.

\begin{lemma}\label{lemma:dminus1deltarich}
Let $\F \subset \binom{[n]}{k}$ be non-trivial $d$-wise intersecting and $X \in \binom{[n]}{d-1}$. Then $d_{\F}^\star(X) < k$.
\end{lemma}

\begin{proof}
Without loss of generality, suppose that $X=[d-1]$ is so that  $d_{\F}^\star(X) \geq k$.  Thus, there exists $\Delta_{k,k} = \{F_1, \ldots, F_k\} \subset \F$ so that $\text{core}(\Delta_{k,k}) = [d-1]$. Next, by the nontriviality of $\F$, for each $j \in [d-1]$, there exists $X_j \in \F$ so that $ j \notin X_j$.

\medskip
\noindent Now, when $d\geq 4$, since $F_1 \cap F_2 = [d-1]$ and $\F$ is $d$-wise intersecting, for $3 \leq m \leq d-1$ $$ F_1 \cap F_2 \cap ( \bigcap\limits_{j \neq m} X_j ) = \{m\}. $$ As a result,  $|X_m \cap [d-1]| = d-2$ and hence  $|X_m \cap (F_j \setminus [d-1])| \geq 1$ for all $j \in [k]$. This yields a contradiction as  \[|X_m| \geq |X_m \cap [d-1]| + \sum_{j=1}^k |X_m \cap (F_j \setminus [d-1])| > k.  \]
When $d=3$, the result follows similarly by considering the cases where $2 \in X_1$ and $2 \notin X_1$. \qedhere
\end{proof}

We are interested in $d$-sets which have large core degree since they intersect elements of our family $\F$ in many places. To this end, we say $D \in \binom{[n]}{d}$ has \textit{large core degree} if $d_{\F}^\star(D) \geq k$.

\begin{lemma}\label{lemma:largeintersectionswithdeltarich}
Let $\F \subset \binom{[n]}{k}$ be non-trivial $d$-wise intersecting and $D \subset [n]$ have large core degree. Then  $|A \cap D| \geq d-1$ for all $A \in \F$. 	
\end{lemma}	

\begin{proof}
Seeking a contradiction, suppose there exists such a $D \subset [n]$ and $A \in \F$ so that $|A \cap D| < d-1$. By definition of large core degree, there exists a Delta system $\Delta_{k,k} = \{F_1, \ldots, F_k \} \subset \F$ so that $\text{core}(\Delta_{k,k}) = D$. By Lemma \ref{lemma:mwiseintersections}, $| A \cap (F_j \setminus D) | \geq (d-1) - |A \cap D|$ for all $j \in [k]$. This is a contradiction as  \[ |A| \geq |A \cap D| + \sum_{j=1}^{k} |A \cap (F_j \setminus D)| >k. \qedhere \]

\end{proof}

Given a family $\F$, we let $\mathcal{S}_d(\F)$ be the possibly empty collection of $d$-sets with large core degree in $\F$. In the proof of Theorem \ref{thm:mainthm1}, we will iteratively find $d$-sets with large core degree and these will always lie within some ground set of size at most $k+1$. As a result, we think of $\mathcal{S}_d(\F) \subset \binom{[k+1]}{d}$. We now note that for $n \geq k(k-d)+d$, the two extremal families $\A(k,d)$ and $\HM(k,d)$ are so that: 
\begin{eqnarray*}
\mathcal{S}_d(\HM(k,d)) &=&  \{ A \in \binom{[k+1]}{d} : [d-1] \subset A \} \\
\mathcal{S}_d(\A(k,d)) &=& K_{d+1}^{(d)}.
\end{eqnarray*}

\begin{lemma}\label{lemma:intersectionofsunrichsets}
Let $\F \subset \binom{[n]}{k}$ be a non-trivial $d$-wise intersecting family. Then $\mathcal{S}_d(\F) \subset \binom{[k+1]}{d}$ is a $(d-1)$-intersecting family.
\end{lemma}

\begin{proof}
Suppose there exists $D_1, D_2 \in \mathcal{S}_d(\F)$ so that $|D_1 \cap D_2| \leq d-2$. By definition, there exists $\Delta_{k,k}^1 = \{ F_1, \ldots, F_k  \} \subset \F$ and  $\Delta_{k,k}^2 = \{ G_1, \ldots, G_k \} \subset \F$ so that $\text{core}(\Delta_{k,k}^i) = D_i$ for $i = 1,2$.  Note that there necessarily exists $F_i \in \Delta_{k,k}^{(1)}$ so that $|(F_i \setminus D_1) \cap D_2| = 0$. By Lemma \ref{lemma:mwiseintersections}, $|(F_i \setminus D_1) \cap (G_j \setminus D_2)| \geq (d-1)- |D_1 \cap D_2|$. This is a contradiction since \[ |F_i| \geq |D_1 \cap D_2| + \sum_{j=1}^k |(F_i \setminus D_1) \cap (G_j \setminus D_2)| > k. \qedhere \]
\end{proof}

As a result of Lemma \ref{lemma:largeintersectionswithdeltarich}, we are interested in $\mathcal{S} \subset \binom{[k+1]}{d}$ which are $(d-1)$-intersecting. The following structural type result yields that the collection of $d$-sets with large core degree is necessarily isomorphic to a subfamily of the collection of $d$-sets with large core degree in the extremal families $\A(k,d)$ and $\HM(k,d)$.

\begin{lemma}\label{lemma:manydeltarichsets1}
If $\mathcal{S} \subset \binom{[k+1]}{d}$ is $(d-1)$-intersecting, then $\mathcal{S}$ is isomorphic to a subfamily of $K_{d+1}^{(d)}$ or $\{ D \in \binom{[k+1]}{d} : [d-1] \subset D \}$.
\end{lemma}

\begin{proof} 
Given distinct $F_1, F_2 \in \mathcal{S}$, we have $|F_1 \cap F_2| =d-1$ and hence without loss of generality, we may assume that $F_1 = [d]$ and $F_2 = [d-1] \cup \{d+1\}$. Now, we let $$\mathcal{S}_1:= \{ F \in \mathcal{S} \setminus \{F_1, F_2\} : |F \cap [d-1]| = d-2   \}$$ and note that if $F \in \mathcal{S}_1$, then $\{d,d+1\} \subset F$ as $\mathcal{S}$ is $(d-1)$-intersecting. We then let $$\mathcal{S}_2:= \{ F \in \mathcal{S} \setminus \{F_1, F_2\} : |F \cap [d-1]| =d-1   \}.$$ For all $F \in \mathcal{S}_1$ and for all $G \in \mathcal{S}_2$, $|F \cap G| = d-2$ and thus if $\mathcal{S}_2 \neq \emptyset$, then $\mathcal{S}_1 = \emptyset$.
\end{proof}

We will now show that if a non-trivial $d$-wise intersecting family $\F \subset \binom{[n]}{k}$ has a particular structure of $d$-sets with large core degree, then $|F| \leq |\HM(k,d)|$.

\begin{lemma}\label{lemma:firststructure} 
Let $\F \subset \binom{[n]}{k}$ be a non-trivial $d$-wise intersecting family. If $\{ A \in \binom{[k]}{d} : [d-1] \subset A\} \subseteq \mathcal{S}_d(\F)$, then $|\F| \leq |\HM(k,d)|$. Moreover, if $|\F| > |\B(k,d)|$, then $\F$ is necessarily isomorphic to some subfamily of $\HM(k,d)$.
\end{lemma}

\begin{proof}
We have that $D_x := [d-1] \cup \{x\} \in \mathcal{S}_d(\F)$ for all $x \in [d,k]$. As a result of Lemma \ref{lemma:largeintersectionswithdeltarich}, for all $A \in \F$, $|A \cap [d-1]| \geq d-2$. We let $\F_1:= \{ A \in \F : |A \cap [d-1]| =d-2  \}$ and $\F_2:= \{ A \in \F : |A \cap [d-1]| =d-1  \}$. Next, using the nontriviality of $\F$, for each $i \in [d-1]$ there exists $X_i \in \F$ so that $i \notin X_i$. We now have two cases based on the collection of sets $\F(\overline{i}) = \{ A \in \F : 1 \notin A \}$ for $i \in [d-1]$.

\medskip
\noindent First, we consider the case where we may find $X_i \in \F(\overline{i})$ so that $$ \bigcap\limits_{i=1}^{d-1} X_i = [d,k].$$ In this case, 
\begin{eqnarray*}
\F_1 &\subseteq&  \{ A \in \binom{[n]}{k} : |A \cap [d-1]| = d-2, [d,k] \subset A \} \\
\F_2 &\subseteq& \{ A \in \binom{[n]}{k} : [d-1] \subset A, A \cap [d,k]  \neq \emptyset  \}.
\end{eqnarray*}

Thus $|\F| = |\F_1| + |\F_2| \leq |\B(k,d)| \leq |\HM(k,d)|$.

\medskip
\noindent Next, if we cannot find $X_i$'s so that they fall in the above case, then without loss of generality$$ \bigcap\limits_{i=1}^{d-1} X_i = [d,k+1].$$  In this case, 
\begin{eqnarray*}
\F_1 &\subseteq&  \{ [k+1] \setminus \{i\} : i \in [d-1]  \} \\
\F_2 &\subseteq& \{ A \in \binom{[n]}{k} : [d-1] \subset A, A \cap [d,k+1]  \neq \emptyset  \}.
\end{eqnarray*}

Thus $|\F| = |\F_1| + |\F_2| \leq |\HM(k,d)|$.
\end{proof}

We now will prove the analog of Lemma \ref{lemma:firststructure} when $\mathcal{S}_d(\F)$ is isomorphic to a subfamily of $K_{d+1}^{(d)}$. 

\begin{lemma}\label{lemma:secondstructure}
Let $\F \subset \binom{[n]}{k}$ be a non-trivial $d$-wise intersecting family. Given that $|\mathcal{S}_d(\F)| \geq 3$ and $\mathcal{S}_d(\F) \subseteq K_{d+1}^{(d)}$, then $\F \subseteq \A(k,d)$.
\end{lemma}

\begin{proof}
Let $D_1, D_2, D_3 \in \mathcal{S}_d(\F)$ be distinct $d$-sets with large core degree. By Lemma \ref{lemma:largeintersectionswithdeltarich} and Lemma \ref{lemma:manydeltarichsets1}, we may assume $D_i= [d+1] \setminus \{i\}$ and that $|A \cap D_i| \geq d-1$ for all $A \in \F$ for $i=1,2,3$. This then implies that $|A \cap [d+1]| \geq d$ for all $A \in \F$ and thus the result follows. 
\end{proof}

\section{ The Delta system method}\label{sec:deltasystemmethod}

The Delta system method is a powerful tool in extremal combinatorics that initially appeared in Deza, Erd{\H o}s and Frankl's \cite{DEF} study of $(n,k,L)$-systems. It has also been used by Frankl and F{\" u}redi \cite{FF} in Chv{\' a}tal's problem of avoiding $d$-simplicies as well as by F{\" u}redi \cite{F1, FJR} on the problem of embedding expansions of forests in $r$-graphs for $r\geq4$.

\subsection{ F{\" u}redi Intersection Semilattice Theorem}

Given a $k$-partite hypergraph $\h \subset \binom{[n]}{k}$ with parts $X_1, \ldots, X_k$, we let the \textit{projection} of $e$ be $$ \text{proj}(e) = \{ i : e \cap X_i \neq \emptyset  \}.$$ Then, let the \textit{intersection pattern} of $e$ on $\h$ be $$ I_{\h}(e) = \{ \text{proj}(g) : g \in \h|_e     \}. $$

We are now able to state F{\" u}redi's Intersection Semilattice lemma. 

\begin{lemma}\label{lemma:semilattice}
For fixed $s,k \in \N$, there exists $c:=c(s,k)$ so that for all $\h \subset \binom{[n]}{k}$ there exists $k$-partite $\h^\star \subset \h$ and $J \subset 2^{[k]}$ so that $|\h^\star| > c|\h|$ and $J$ is intersection closed. Moreover, for all $e \in \h^\star$ that $I_{\h^\star}(e) = J$ and that for all $f \in \h^\star|_e$, $d_{\h^\star}^\star(f) \geq s$.
\end{lemma}

Given $\h \subset \binom{[n]}{k}$ so that it satisfies the conclusions of Lemma \ref{lemma:semilattice}, we say that $\h$ is $(s,J)$-homogeneous and if there exists an $s \in \N$ so that $\h$ is $(s,J)$-homogeneous, we say that $\h$ is $J$-homogeneous.

Let $J \subset 2^{[k]}$, then the \textit{rank} of $J$, denoted $\rho(J)$ is defined as $$\rho(J) := \min \{ |e| : e \subset [k]: d_J(e) = 0 \}.$$

\subsection{ Application to Non-trivial $d$-wise Intersecting Families}

\noindent Let $\F \subset \binom{[n]}{k}$ be non-trivial and $d$-wise intersecting. Applying Lemma \ref{lemma:semilattice} with $s=k$ to large subfamilies $ \h \subset \F$ when $n$ sufficiently large gives a particular intersection structure $J$. To this end, let $n_0(k,d):= d+e(k^22^k)^{2^k}(k-d)$ and we let $c_k:=c(k,k) > (k^22^k)^{-2^k}$.

\begin{lemma}\label{lemma:keylemma}
Let $\h \subset \F$ be $J$-homogeneous and so that $|\h| \geq \binom{n-d}{k-d}$ where $n> n_0(k,d)$. Then $$J \subset \bigcup\limits_{l=d}^{k-1} \binom{[k]}{l}.$$ Moreover, $|J \cap \binom{[k]}{d}| = 1$.
\end{lemma}

\begin{proof} 
Lemma \ref{lemma:mwiseintersections} gives that $|A \cap B| \geq d-1$ for all $A, B \in \F$ which yields that  $$J \subset \bigcup\limits_{l=d-1}^{k-1} \binom{[k]}{l}.$$
Lemma \ref{lemma:dminus1deltarich} yields that $J \cap \binom{[k]}{d-1} = \emptyset$ and hence  $$J \subset \bigcup\limits_{l=d}^{k-1} \binom{[k]}{l}.$$

Now, since $J$ is intersection closed, $|J \cap \binom{[k]}{d}| \leq 1$. If $J \cap \binom{[k]}{d} = \emptyset$, then without loss of generality, suppose that $[d+1]$ is the inclusion minimal element of $J$. Suppose there exists $X \in J$ so that $[d+2,k] \subset X$, then there is an $i \in [d+1]$ so that $i \notin X$ as $[k] \notin J$. Now, $X \cap [d+1] \in J$, but $|X \cap [d+1]| <(d+1)$ so we necessarily have that $d_J([d+2,k])=0$. As a result, $\rho(J) \leq k-d-1$. Next, if $\h^\star \subset \binom{[n]}{k}$ is $J$-homogeneous with $\rho(J)= r$, then $|\h^\star| \leq \binom{n}{r}$ (see \cite{MV}). Thus,  $|\h^\star| \leq \binom{n}{k-d-1}$. However, for $n> n_0(k,d)$, \[ |\h^\star| > c_k |\h| \geq c_k \binom{n-d}{k-d} >\binom{n}{k-d-1}. \qedhere \]
\end{proof}

\section{Proof of Theorem \ref{thm:mainthm1}}\label{sec:proof}

In this section, we will prove Theorem \ref{thm:mainthm1} by repeated application of Lemma \ref{lemma:keylemma} and the structural results from Section \ref{sec:structuredeltarich}. 

\begin{proof}[Proof of Theorem \ref{thm:mainthm1}]
As a result of Lemma \ref{lemma:firststructure} and Lemma \ref{lemma:secondstructure}, it suffices to show that $|\mathcal{S}_d(\F)|\geq 3$ and that we either have $\mathcal{S}_d(\F)$ contains $\{ A \in \binom{[k]}{d} : [d-1] \subset A \}$ with $|\mathcal{S}_d(\F)| = k-d+1$ or $\mathcal{S}_d(\F)$ is isomorphic to a subfamily of $K_{d+1}^{(d)}$. An application of Lemma \ref{lemma:keylemma} yields a $d$-set $D_1$ which has large core degree. We now consider $$\h_1: = \{ A \in \F : D_1 \nsubseteq A \}$$ and again applying Lemma \ref{lemma:keylemma} yields a $d$-set $D_2$ which has large core degree and $D_1 \neq D_2$. We can iteratively apply Lemma \ref{lemma:keylemma} $s$ times to get $\{D_1, \ldots, D_s\} \in \mathcal{S}_d(\F)$ where the particular value of $s$ depends on $|\F|$. 

\medskip 
In the case where $2d<k$, we may suppose that $|\F|> |\B(k,d)| > |\A(k,d)|$ and we take $s=k-d+1$. Noting that $s>d+1$ then yields that $\mathcal{S}_d(\F)$ is not isomorphic to a subfamily of $K_{d+1}^{(d)}$. Lemma \ref{lemma:firststructure} then yields that $\F$ is isomorphic to a subfamily of $\HM(k,d)$.

\medskip 
In the case where $2d=k$, we may suppose that $|\F|> |\A(k,d)| > |\B(k,d)|$ and also take $s=k-d+1$ where we note $k-d+1=d+1$. Noting that $|\F| > |\A(k,d)|$ then yields that $\mathcal{S}_d(\F)$ is not isomorphic to  $K_{d+1}^{(d)}$. Lemma \ref{lemma:firststructure} then yields that $\F$ is isomorphic to a subfamily of $\HM(k,d)$.

\medskip 
In the case where $2d \geq  k+1$, we may suppose that $|\F|> |\HM(k,d)| > |\B(k,d)|$. When $d<k-1$, we take $s=k-d+1 \geq 3$ and when $d=k-1$, and a Inclusion-Exclusion argument yields that we may take $s=3$. In both of these cases, noting that $|\F| > |\HM(k,d)|$ then yields that $\mathcal{S}_d(\F)$ is not isomorphic to a subfamily of $\{ A \subset [k]: [d-1] \subset A\}$. Lemma \ref{lemma:secondstructure} yields that $\F$ is isomorphic to a subfamily of $\A(k,d)$.
\end{proof}

\section{Concluding Remarks}

The general framework of considering $d$-sets with large core degree from Section \ref{sec:structuredeltarich} requires that $n \geq k(k-d)+d$. One can probably alter the threshold of $d$-sets with large core degree to $k-d+3$ as opposed to $k$ via similar arguments in this paper to slightly improve our value of $n_0(k,d)$, but it would still be doubly exponential. 

\medskip 
For the particular cases of $d=k-1$ and $d=k-2$, we can achieve more reasonable values on $n_0(k,d)$ via direct arguments. In the case where $d=k-2$, Lemma \ref{lemma:mwiseintersections} yields that $|A \cap B| \geq k-3$ for all $A, B \in \F$. By considering the trace of $\F$ on $A$ for a fixed $A \in \F$,  $$ |\F| = \sum_{ X \in \F|_A} |H_X|$$ where we let $H_X = \{ Y \subset [n] : X \sqcup Y \in \F \}$ be the link hypergraph of $X$ in $\F$. If $|X_1 \cap X_2| = m$, then the link hypergraphs $H_{X_1}$ and $H_{X_2}$ are necessarily $(d-1)-m$ cross intersecting. Moreover, Frankl's improved bounds \cite{F3} on the Erd{\H o}s Matching conjecture give that if $|H_X|$ is sufficiently large, then $\nu(H_X) = d_{\F}^\star(X) = s$. Using these facts, we consider various cases and in each case prove that if $n\geq 7k$, then $$ |\F| = \sum_{ X \in \F|_A} |H_X| \leq |\A(k,d)|.$$

\medskip 
Given a family $\F \subset 2^{[n]}$, let $\delta(\F) := \min_{i \in [n]} |\{ A \in \F : i \in A\}|$ be the minimum degree of an element in $\F$. Recently, Frankl, Han, Huang, and Zhao \cite{FHHZ} proved a degree version of  Theorem \ref{thm:HM}. It would be interesting to see if a degree version of our result would hold:  

\medskip 
Does there exist $n_1(k,d)$ so that for $n>n_1(k,d)$, if $\F \subset \binom{[n]}{k}$ is a non-trivial $d$-wise intersecting family, then $\delta(\F) \leq \max \{ \delta(\HM(k,d)), \delta(\A(k,d)) \}$? 

\medskip
Moreover, it would be interesting to see how an optimal $n_0(k,d)$ compares with optimal an $n_1(k,d)$. In the case where $d=2$, we have that $n_0(k,2) = 2k+1$ and $n_1(k,2)= O(k^2)$ although Frankl et al. \cite{FHHZ} asked if their degree Hilton-Milner holds for $n \geq 2k+1$.


\medskip
\medskip

\begin{thebibliography}{99}
	
\bibitem{BD} A.Brace, B.A. Daykin, A finite set covering theorem, Bulletin of the Australian Mathematical Society, 5(2), 1971, 197-202. 

\bibitem{DEF} M. Deza, P. Erd{\H o}s, P. Frankl, Intersection properties of systems of finite sets, Proc. London Math. Soc., 36(3), 1978, 369-384.

\bibitem{EKR} P.Erd{\H o}s. C. Ko , R. Rado, Intersection theorems for systems of finite sets, The Quarterly Journal of Mathematics, 48(12), 1961.
	
\bibitem{F} P. Frankl, On intersecting families of finite sets, J. Combin. Theory Ser. A, 24(2), 1978, 146-161.

\bibitem{F2} P. Frankl, On a problem of Chv{\' a}tal and Erd{\H o}s on hypergraphs containing no generalized simplex, J. Combin. Theory Ser. A, 30(2), 1981, 169-182.

\bibitem{F3} P. Frankl, Improved bounds on the Erd{\H o}s matching conjecture, J. Combin. Theory Ser. A 120(5), 2013, 1068-1072.

\bibitem{FF} P.Frankl, Z.F{\" u}redi, Exact Solution of some Tur{\' a}n type problems, J. Combin. Theory Ser. A 45(2), 1987, 226-262. 

\bibitem{FT} P. Frankl, N. Tokushige, Extremal problems for finite sets, American Mathematical Soc., 2018.

\bibitem{FHHZ} P. Frankl, H. Huang,  J. Han Y. Zhao, A degree version of the Hilton-Milner Theorem, J. Combin. Theory Ser. A, 155, 2018, 493-502.

\bibitem{F1} Z. F{\" u}redi, Linear trees in uniform hypergraphs. European J. Combin. Theory, 35, 2014, 264-272.

\bibitem{FJR} Z. F{\" u}redi, T. Jiang, R. Seiver, Exact solution of the hypergraph Tur{\' a}n problem for k-uniform linear paths, Combinatorica, 34(3), 2011.


\bibitem{FS} Z. F{\" u}redi, B. Sudakov, Extremal set systems with restricted k-wise intersections, J. Combin. Theory Ser. A, 105(1), 2004, 143-159.

\bibitem{GS} V. Grolmusz , B.Sudakov, On k-wise set-intersections and k-wise Hamming-distances, J. Combin. Theory Ser. A, 99(1), 2002, 180-190. 

\bibitem{HY} J. Han, Y. Kohayakawa, The maximum size of a non-trivial intersecting uniform family that is not a subfamily of the Hilton--Milner family, Proc. Amer. Math. Soc., 145, 2017, 73-87.

\bibitem{HM} A.W. J. Hilton, E.C. Milner, Some intersection theorems for systems of finite sets, . Quart. J. Math. Oxford Ser., 18 (2), 1967, 369-384. 

\bibitem{K} A. Kupavskii, Structure and properties of large intersecting families, \href{https://arxiv.org/abs/1810.00920}{https://arxiv.org/abs/1810.00920}, 2019.

\bibitem{MV} D. Mubayi, J. Verstraete, A survey of Turán problems for expansions. In: Beveridge A., Griggs J., Hogben L., Musiker G., Tetali P. (eds) Recent Trends in Combinatorics. The IMA Volumes in Mathematics and its Applications, 159, 2016.

	
\end{thebibliography}
\end{document}